\documentclass[12pt]{amsart}

\usepackage{url}
\usepackage{graphicx}
\usepackage{amsmath}
\usepackage{amscd}
\usepackage{amsfonts}
\usepackage{amssymb}
\usepackage{color}
\usepackage{fullpage}
\usepackage{mathtools}
\usepackage[pagebackref,hypertexnames=false, colorlinks, citecolor=red,linkcolor=blue, urlcolor=red]{hyperref}
\usepackage{todonotes}
\usepackage{enumitem}
\usepackage{comment}

\numberwithin{equation}{section}

\newtheorem{theorem}{Theorem}[section]
\newtheorem{lemma}[theorem]{Lemma}
\newtheorem{proposition}[theorem]{Proposition}

\newtheorem{quest}[theorem]{Question}

\newtheorem{corollary}[theorem]{Corollary}

\theoremstyle{definition}
\newtheorem{example}[theorem]{Example}
\newtheorem{remark}[theorem]{Remark}
\newtheorem{definition}[theorem]{Definition}

\newcommand{\be}{\begin{equation}}
\newcommand{\ee}{\end{equation}}
\newcommand{\bes}{\begin{equation*}}
\newcommand{\ees}{\end{equation*}}

\newcommand{\cA}{\mathcal{A}}

\newcommand{\cD}{\mathcal{D}}
\newcommand{\cE}{\mathcal{E}}

\newcommand{\cH}{\mathcal{H}}

\newcommand{\cK}{\mathcal{K}}
\newcommand{\cL}{\mathcal{L}}

\newcommand{\bB}{\mathbb{B}}
\newcommand{\bC}{\mathbb{C}}
\newcommand{\C}{\mathbb{C}}
\newcommand{\bD}{\mathbb{D}}

\newcommand{\bF}{\mathbb{F}}
\newcommand{\bM}{\mathbb{M}}
\newcommand{\bN}{\mathbb{N}}

\newcommand{\ol}{\overline}

\newcommand{\re}{\operatorname{Re}}
\newcommand{\im}{\operatorname{Im}}

\newcommand{\diag}{\operatorname{diag}}

\newcommand{\id}{\operatorname{id}}

\newcommand{\spn}{\operatorname{span}}

\newcommand{\fB}{{\mathfrak{B}}}
\newcommand{\fC}{{\mathfrak{C}}}
\newcommand{\fD}{{\mathfrak{D}}}

\newcommand{\FORAL}{\text{ for all }}

\newcommand{\dom}{\operatorname{Dom}}
\newcommand{\GL}{\mathrm{GL}}

\begin{document}

\title{A spectral radius for matrices over an operator space}

 \author{Orr Moshe Shalit}
 \address{O.M.S., Technion Israel Institute of Technology\\
 Technion City, Haifa\; 3200003\\
 Israel}

 \author{Eli Shamovich}
 \address{E.S., Department of Mathematics\\
 Ben-Gurion University of the Negev\\
 Beer-Seva\; 8410501\\
 Israel}

 \thanks{The work of O.M. Shalit is partially supported by ISF Grant no.\ 431/20. The work of E. Shamovich is partially supported by BSF grant no.\ 2022235.
 }

\addcontentsline{toc}{section}{Abstract}

\begin{abstract}
With every operator space structure $\cE$ on $\bC^d$, we associate a spectral radius function $\rho_\cE$ on $d$-tuples of operators. 
For a $d$-tuple $X = (X_1, \ldots, X_d) \in M_n(\bC^d)$ of matrices we show that $\rho_\cE(X)<1$ if and only if $X$ is jointly similar to a tuple in the open unit ball of $M_n(\cE)$, that is, there is an invertible matrix $S$ such that $\|S^{-1}X S\|_{M_n(\cE)}<1$, where $S^{-1} X S =(S^{-1} X_1 S, \ldots, S^{-1} X_d S)$. 
More generally, for all $X\in B(\cK) \otimes_{\min} \cE$ we show that $\rho_\cE(X) < 1$ if and only if there exists an invertible $S\in B(\cK) \otimes I$ such that $\|S^{-1} X S\|<1$. 
When $\cE$ is the row operator space, for example, our spectral radius coincides with the joint spectral radius considered by Bunce, Popescu, and others, and we recover the condition for a tuple of matrices to be simultaneously similar to a strict row contraction. 
When $\cE$ is the minimal operator space $\min(\ell^\infty_d)$, our spectral radius $\rho_\cE$ is related to the joint spectral radius considered by Rota and Strang but differs from it and has the advantage that $\rho_\cE(X)<1$ if and only if $X$ is simultaneously similar to a tuple of strict contractions. 
We show that for an nc rational function $f$ with descriptor realization $(A,b,c)$, the spectral radius $\rho_\cE(A)<1$ if and only the domain of $f$ contains a neighborhood of the noncommutative closed unit ball of the operator space dual $\cE^*$ of $\cE$. 
\end{abstract}

\maketitle

\section{Introduction}\label{sec:intro}

\subsection{Overview} 
The main goal of this paper is the following: {\em Given an operator space $\cE$ and a matrix $X \in M_n(\cE)$ over $\cE$, determine whether there exists a scalar matrix $S \in M_n$ such that $\|S^{-1} X S\|_{M_n(\cE)} < 1$.}
Consider the case $\cE = \bC$, when $X$ is just an $n \times n$ matrix of scalars. 
Clearly, a necessary condition is that the spectrum $\sigma(X)$ of $X$ be contained in the unit disc $\bD$. Considering the upper triangular form of $X$, it is not hard to see that this condition is also sufficient. 
Said differently, $X$ is similar to a strict contraction if and only if the {\em spectral radius} $r(X) = \max\{|\lambda| : \lambda \in \sigma(X)\}$ is strictly less than one. 
The analogous result for operators on an infinite dimensional space was obtained by Rota, as a consequence of his result that every operator with spectral radius strictly less than one is similar to the restriction of the unilateral shift of infinite multiplicity to an invariant subspace \cite{Rota60}.

For a general $\cE$, the situation becomes more interesting. 
Of course, $X$ can be considered to be an operator on {\em some} Hilbert space, so in principle, Rota's theorem applies, but note that our goal is not to find just any operator that induces some similarity but to find a scalar matrix $S$ that does the job, respecting the operator space structure. 
For example, when $\cE = \bC^d$, then $X \in M_n(\bC^d) \cong M_n(\bC)^d$ is a tuple of matrices $X = (X_1, \ldots, X_d)$, and then by {\em similarity} we mean $S^{-1} X S = (S^{-1} X_1 S, \ldots, S^{-1} X_d S)$, that is, a joint similarity of the separate matrices. 
Characterizing when this is possible is the main challenge that we address in this paper. 

The spectrum of an operator is a notion whose generalization to tuples of operators is not immediate. 
By Gelfand's spectral radius formula, if we define the quantity
\be\label{eq:spectral_radius}
\rho(X) := \lim_n \|X^n\|^{1/n} ,
\ee
then $r(X) = \rho(X)$ \cite[Section I.11]{Gelfand}. 
We can reformulate that a necessary and sufficient condition for $X \in M_n(\bC)$ to be similar to a strict contraction is $\rho(X)<1$. 

A satisfactory solution to our problem, extending the above characterization of when an operator is similar to a strict contraction, was obtained in the case that $\cE$ is the row operator space. 
For $X \in B(\cK)^d$, the {\em joint spectral radius} $\rho(X)$ is defined to be
\begin{equation}\label{eq:Pop_spec_rad}
\rho(X) = \lim_{n \to \infty} \left\| \sum_{|w| = n} X^{w} (X^{w})^* \right\|^{1/2n}.
\end{equation}
This quantity first appeared in a paper by Bunce \cite{Bunce}; it is a particular instance of a more general spectral radius defined later by Popescu \cite{Pop14}. 
Popescu showed that $\rho(X) < 1$ if and only if $X$ is jointly similar to a strict row contraction.  
This result was rediscovered in \cite{SSS20} (see also \cite{Pascoe21}), where the joint spectral radius was further studied and applied. It was shown that the joint spectral radius is log-subharmonic on discs and that for a tuple of matrices $X \in M_n^d$ the joint spectral radius is $\rho(X) = \max\{\rho(X^{(1)}) \ldots, \rho(X^{(k)})\}$ where $X^{(1)}, \ldots, X^{(k)}$ are the Holder-Jordan components of $X$.

\begin{remark}\label{rem:matrix}
It is interesting to note that $\rho(X) \leq 1$ does not imply that $X$ is similar to a row contraction; indeed, consider the case $d = 1$ and 
\[
X = \begin{pmatrix} 1 & 1 \\ 0 & 1 \end{pmatrix} .
\]
Clearly, the spectral radius $\rho(X)$ is equal to $1$, but any similarity will take $X$ to an operator unitary equivalent to an upper triangular matrix with $1$s on the diagonal and a nonzero entry in the upper right corner. 
Thus, any operator similar to $X$ must have norm strictly greater than $1$. 
However, $\rho(X)$ is always equal to $\inf\{\|S^{-1} X S\| : S \in \GL_n\}$. 
\end{remark}

The joint spectral radius \eqref{eq:Pop_spec_rad} gives a concrete numerical invariant that determines whether a tuple is similar to an element of the row ball $\fB_d$ (see Example \ref{ex:row_ball} below). 
We shall show that an appropriate adaptation of the formula \eqref{eq:Pop_spec_rad} gives rise to a spectral radius $\rho_\cE$ such that for all $X \in M_n(\cE)$, $\rho_\cE(X)<1$ if and only if there exists $S \in M_n$ such that $\|S^{-1} X S\|_{M_n(\cE)} < 1$. In fact, our definition also works in infinite dimensional spaces, and we obtain a family of numerical invariants that characterize when a tuple of bounded operators is similar to a strict row contraction, or to a tuple of strict contractions, and so forth. 

In Definition \ref{def:gen_SpecRad}, we define the spectral radius $\rho_\cE$ corresponding to an operator space $\cE$. 
Briefly, for $X \in B(\cK) \otimes_{\min} \cE$, the spectral $\rho_\cE(X)$ is defined to be the ordinary spectral radius \eqref{eq:spectral_radius} of $X$ when considered as an element of the C*-algebra $B(\cK) \otimes_{\min} C^*_{\max}(\cE)$ where $C^*_{\max}(\cE)$ is the maximal C*-cover of $\cE$. For $X \in M_n(\cE)$, the spectral radius $\rho_\cE(X)$ is just spectral radius of $X$ when considered as an element of the matrix algebra $M_n(C^*_{\max}(\cE))$ over the maximal C*-algebra $C^*_{\max}(\cE)$ generated by $\cE$.

When $\cE$ is a finite-dimensional operator space, which can be described by a family of matrix norms on $\bC^d$, we obtain a concrete spectral radius for tuples of matrices defined by an explicit formula \eqref{eq:Hsr}. 
For the well-studied case, when $\cE$ is the row operator space, our new spectral radius $\rho_\cE$ coincides with the joint spectral radius \eqref{eq:Pop_spec_rad} mentioned above. 
We note that several generalizations of the joint spectral radius have been suggested in the literature (e.g. \cite{bhat2024spectral,Pop14}), and it seems that our notion differs from them. 

Our main result is Theorem \ref{thm:gen_SpecRad}, which says that when $X \in M_n^d$ is a $d$-tuple of operators, then $\rho_\cE(X)<1$ if and only if $X$ is jointly similar to an element in the corresponding open unit ball. 
Again, when $\cE$ is the row operator space structure on $\bC^d$, we recover the known necessary and sufficient condition for a tuple of matrices to be jointly similar to a strict row contraction. 
When $\cE$ is the minimal operator space $\min(\ell^\infty_d)$ over $\ell^
\infty_d$, we obtain a condition for a tuple to be simultaneously similar to a tuple of strict contractions. 
Determining when a tuple is similar to a strict row contraction or to a tuple of strict contractions is an interesting problem in its own right (see \cite{ando1998simultaneous}). Still, we point out that it also has significant applications in pure and applied mathematics \cite{daubechies1992sets,daubechies1992two}. 
As another modest application, we note that the {\em similarity envelope} of the row ball, i.e., all tuples simultaneously similar to a strict row contraction, played a role in the classification of algebras of bounded nc functions on subvarieties of the row ball \cite{SSS20}. 
We hope that understanding joint similarities into general operator balls will help with the classification of the algebras of bounded nc functions on subvarieties of nc operator balls studied in \cite{SS25,sampat2024weak}. 

In Section \ref{sec:nc_rat}, we apply our spectral radius to the study of nc rational functions. 
In Theorem \ref{thm:bdd_rat_func}, we show that for an nc rational function $f$ with descriptor realization 
\[
f(X) = (c \otimes I)^* \left(I - \sum_{j=1}^d A_j \otimes X_j \right)^{-1} (b \otimes I)
\]
a necessary and sufficient condition for the domain of $f$ to contain an open nc operator ball corresponding to $\cE$ of radius $R > 1$ is that $\rho_{\cE^*}(A)<1$, where $\cE^*$ is the operator space dual of $\cE$. 
We then close the paper by working out a couple of illustrative examples. 
In the case that $\cE$ is the row ball, it was shown in \cite{JMS-rat} that $\rho_{\cE^*}(A)<1$ is a necessary condition for $f$ to be bounded on the unit ball $\bB_\cE$, but in Section \ref{subsec:case} we show by an example of an nc rational function on the nc polydisc that it is not necessary in general. 
Finally, in Section \ref{subsec:switch}, we explore the possibility of switching the roles of $\cE$ and $\cE^*$ and obtain rational functions on different domains. 

\subsection{Preliminaries and notation}\label{sec:pre}

In this paper a tensor product $A \otimes B$ between operators $A \in B(\cK)$ and $B \in B(\cH)$ is to be understood, unless specifically indicated otherwise,  in the spatial sense 
\[
A \otimes B \in B(\cK) \ol{\otimes} B(\cH) \cong B(\cK \otimes \cH). 
\]
When $A$ and $B$ are elements of operator spaces $\cD \subseteq B(\cK)$ and $\cE\subseteq B(\cH)$, we understand $A \otimes B$ as an element of the minimal tensor product 
\[
\cD \otimes \cE = \cD \otimes_{\min} \cE \subseteq B(\cK \otimes \cH).
\]

For an operator space $\cE$, we define the {\em noncommutative (nc) operator ball} $\bB_\cE$ to be 
\begin{equation}\label{eq:op_ball}
\bB_\cE = \cup_{n=1}^\infty \{X \in M_n(\cE) : \|X\|_n<1\}.
\end{equation}
An operator space can always be represented concretely as a subspace $\cE \subseteq B(\cH)$. 
When $\cE$ is finite-dimensional, then it can be represented in an even more concrete manner as follows. 

Consider the ``nc universe" consisting of all $d$-tuples of $n \times n$ matrices
\[
\bM^d = \sqcup_{n=1}^\infty M_n^d .
\]
Given a $d$-dimensional operator space $\cE \subseteq B(\cH)$ with basis $Q_1, \ldots, Q_d$ fixed once and for all, we obtain a concrete (and coordinate dependent) presentation of the operator ball \eqref{eq:op_ball} via the linear operator-valued polynomial $Q(z) = \sum_{j=1}^d z_j Q_j$. 
This polynomial can be evaluated on a $d$-tuple $X = (X_1, \ldots, X_d)$ of matrices by 
\[
Q(X) = \sum_{j=1}^d X_j \otimes Q_j. 
\]
The \emph{nc operator ball} $\bD_Q$ giving a coordinate representation of $\bB_\cE$ is defined as 
\begin{equation} \label{eq:operator_ball}
\bD_Q = \{ X \in \bM^d : \| Q(X) \| < 1 \}.
\end{equation} 
For any $n \in \bN$, $\bD_Q(n) := \bD_Q \cap M_n^d$ can be identified with the open unit ball of $M_n(\cE)$ via 
\begin{equation} \label{eq:identify_DQ}
\bD_Q(n) \ni X \longleftrightarrow Q(X) \in M_n(\cE) \subset M_n(B(\cH))
\end{equation}
and in this way $\bD_Q$ is identified with $\bB_\cE$, 
where an element $X \in \bD_Q(n)$ is endowed with the norm $\|X\| :=\|Q(X)\|$. 
In fact, if we identify $\cE$ with the space $\bC^d$ endowed with the operator space structure given by $Q$, then $\bD_Q$ {\em is} $\bB_\cE$. 
It is convenient to have both points of view available and we shall switch between them freely. 

For notational convenience, we add an ``infinite level" to $\bD_Q$ by defining the proper class
\begin{equation}\label{eq:DQinflevel}
\bD_Q(\infty) = \left\{ X \in B(\cK)^d \textrm{ for some Hilbert space } \cK : \left\|\sum_j X_j \otimes Q_j \right\| < 1 \right\}
\end{equation}
and we define the generalized nc operator unit ball $\bD_Q^{(\infty)}$ as the disjoint union
\begin{equation}\label{eq:DQinf}
\bD_Q^{(\infty)} = \bD_Q \sqcup \bD_Q(\infty).
\end{equation}

The following examples serve to illustrate the notions just introduced. 
We shall refer to them repeatedly later in this paper.
\begin{example}\label{ex:row_ball}
The  {\em row ball} $\fB_d$ is defined to be $\bB_\cE$ when $\cE=\bC^d$ with the row operator space structure
\begin{equation}\label{eq:Ball}
\fB_d = \left\{ X \in \bM^d : \left\|\sum X_j X_j^*\right\| < 1\right\} .
\end{equation} 
Concretely, $\fB_d = \bD_Q$ where $Q_j = E_{1j} \in B(\bC^d)$ is the rank one operator $E_{1j} = e_1 \otimes e_j^* \colon h \mapsto \langle h, e_j \rangle e_1$, where $e_j$ denotes the $j$th standard basis vector in $\bC^d$; 
that is, $E_{1j}$ is the operator represented in the standard basis by the matrix with $1$ in the $j$th slot of the first row and zeros elsewhere. The infinite row ball $\fB_d^{(\infty)}$ denotes the class of all strict row contractions acting on some Hilbert space. 

Similarly, the column ball 
\[
\fC_d = \left\{X \in \bM^d : \left\|\sum X_j^* X_j\right\| < 1 \right\}
\]
is the nc operator ball corresponding to the column operator space, and consists of all strict column contractions acting on a finite dimensional space. 
\end{example}

\begin{example}\label{ex:nc_polydisc}
If $Q_{j} = E_{jj} \in B(\bC^d)$ is the rank one projection $E_{jj} = e_j \otimes e_j^* \colon h \ni \bC^d \to \langle h, e_j \rangle e_j \in \bC^d$ onto the $j$th basis vector, then we have that $Q(z) = \sum_{j=1}^d z_j Q_j = \diag(z_1, \dots, z_d)$, and the resulting nc operator ball $\bD_Q$ is the {\em nc unit polydisc}:
\[
\fD_d = \left\{X \in \bM^d : \|X\|_{\infty} := \max_{1\leq j\leq d} \|X_j\| <1 \right\}.
\]
The nc unit polydisc is the nc operator unit ball corresponding to the minimal operator space $\min(\ell^\infty_d)$ over the Banach space $\ell^\infty_d = (\bC^d, \|\cdot\|_\infty)$. 
\end{example}

\begin{example}\label{ex:diamond}
Let $U_1, \ldots, U_d$ be the unitaries that generate the full group C*-algebra of the free group $C^*(\bF_d)$. 
By \cite[Theorem 9.6.1]{PisierBook}, $\spn\{U_1, \ldots, U_d\}$ is the maximal operator space $\max(\ell^1_d)$ over the Banach space $\ell^1_d = (\bC^d, \|\cdot\|_1)$, which is also the operator space dual $(\min(\ell^\infty_d))^*$. 
We define the {\em nc diamond} to be 
\[
\fD_d^\circ = \left\{X \in \bM^d : \left\| \sum_{j=1}^d X_j \otimes U_j \right\| < 1 \right\}. 
\]
The notation derives from the fact that an equivalent way to define the nc diamond is as the {\em polar dual} of the nc polydisc: 
\[
\fD_d^\circ = (\fD_d)^\circ := \left\{X \in \bM^d : \left\| \sum X_j \otimes A_j \right\| < 1 , \textrm{ for all } A \in \fD_d\right\} .
\]
\end{example}

In this paper, we shall use extensively bounded nc functions on nc operator balls. 
An essentially self-contained presentation of everything we require can be found in \cite{SS25}, including pointers to references to where a thorough introduction to nc function theory can be found. 
Here, we only briefly sketch the tools that we need.  
Given an operator space $\cE = \spn\{Q_1, \ldots, Q_d\}$ and a corresponding nc operator ball $\bD_Q$, 
recall that a function $f \colon \bD_Q \to \bM^1$ is an nc function if it is 
\begin{enumerate}
\item graded: $f(X) \in M_n$ if $X \in M_n^d$; 
\item respects direct sums: $f(X \oplus Y) = f(X) \oplus f(Y)$; 
\item respects similarities: $f(S^{-1} XS) = S^{-1}f(X) S$. 
\end{enumerate}
Let $H^\infty(\bD_Q)$ denote the algebra of bounded nc functions on $\bD_Q$ equipped with the norm
\[
\|f\| = \sup_{X \in \bD_Q} \|f(X)\| . 
\]
We let $A(\bD_Q)$ denote the closure of nc polynomials in $H^\infty(\bD_Q)$. 
Then $A(\bD_Q)$ consists precisely of the bounded nc functions that are uniformly continuous on $\ol{\bD_Q}$. 
Here, uniform continuity can be interpreted in the most naive way: an nc function $f \colon \bD_Q \to \bM^1$ is said to be uniformly continuous if, for every $\epsilon>0$, the exists $\delta > 0$, such that for all $n$ and all $X, Y \in \bD_Q(n)$, 
\[
\|X - Y\| < \delta \Longrightarrow \|f(X) - f(Y)\| < \epsilon . 
\]
The algebra $A(\bD_Q)$ has a natural operator algebra structure, and it can be identified with the universal unital operator algebra $OA_u(E)$ generated by the operator space dual $E = \cE^*$ of $\cE$ (see \cite[Section 7]{SS25}). The finite-dimensional, completely contractive representations of $A(\bD_Q)$ are in one-to-one correspondence with the points in $\ol{\bD_Q}$, where every $X \in \ol{\bD_Q}$ gives rise to a point evaluation $\pi_X \colon f \mapsto f(X)$. 

Every $f \in H^\infty(\bD_Q)$ has a convergent Taylor-Taylor series around $0$
\[
f(X) = \sum_{n=0}^\infty f_n(X) = \sum_{n=0}^\infty \sum_{|w|=n} c_w X^w \quad , \quad X \in \bD_Q, 
\]
where the projection onto the {\em $n$th homogeneous component} $f_n(X) = \sum_{|w|=n} c_w X^w$ is completely contractive, in particular $\|f_n\| \leq \|f\|$. 
The series $f = \sum_n f_n$ converges pointwise to $f$, and if $f \in A(\bD_Q)$, then its Ces\`{a}ro means converge in norm. 

Finally, one can also consider operator-valued (or matrix-valued) nc functions, which are graded, direct sum and similarity preserving functions $\bD_Q \to \bM(B(\cH)) := \sqcup_{n=1}^\infty M_n(B(\cH))$, where $\cH$ is some Hilbert space. Similar definitions and results hold in this case. 
In particular, we shall need the following proposition. 

\begin{proposition}\label{prop:series}
If $f \colon \bD_Q \to \bM(M_k)$ is a bounded nc function, then $f$ has a convergent power series 
\[
f(X) = \sum_{n=0}^\infty f_n(X) = \sum_{n=0}^\infty \sum_{|w|=n} C_w \otimes X^w \quad , \quad X \in \bD_Q, 
\]
for some matrix coefficients $C_w \in M_k(\bC)$, and the {\em $n$th homogeneous component} $f_n(X) = \sum_{|w|=n} C_w \otimes X^w$ is bounded: $\|f_n\| \leq \|f\|$. 
The series $\sum_n f_n$ converges absolutely and uniformly on $r\bD_Q$ for every $r < 1$, and hence $f$ is uniformly continuous on $r\bD_Q$. 
\end{proposition}
\begin{proof}
The first part follows from \cite[Proposition 3.5]{SS25}. 
The uniform convergence and continuity then follows readily from $\|f_n\|\leq \|f\|$, because 
\[
f(X) = \sum_{n=0}^\infty r^n f_n(X/r)
\] 
converges absolutely and uniformly on $r\bD_Q$. 
\end{proof}

\section{Joint spectral radii for general nc operator balls}\label{sec:spectral_rad}

Every operator space $\cE$ can be embedded in a universal unital operator algebra $OA_u(\cE)$ which is generated by $\cE$ as an operator algebra; for a definition of $OA_u(\cE)$, see \cite[pp. 112--113]{PisierBook} (as mentioned in the introduction, by \cite[Section 7]{SS25}, this universal operator algebra can also be defined as the algebra $A(\bB_{\cE^*})$ of uniformly continuous nc functions on $\bB_{\cE^*}$). 
Furthermore, $\cE$ can be embedded in a universal maximal unital C*-cover $C^*_{\max}(\cE)$, and $OA_u(\cE)$ can be identified with the unital norm closed algebra generated by $\cE$ in $C^*_{\max}(\cE)$; see \cite[Remark 8.16]{PisierBook}. 
Let $\iota \colon \cE \to OA_u(\cE) \subset C^*_{\max}(\cE)$ denote the completely isometric embedding of $\cE$ in $OA_u(\cE)\subset C^*_{\max}(\cE)$. 
For brevity, we shall use $\iota$ to denote also the ampliations 
\[
\id_{B(\cK)} \otimes \iota \colon B(\cK) \otimes_{\min} \cE \to B(\cK) \otimes_{\min} C^*_{\max}(\cE), 
\]
where $\cK$ is some Hilbert space. 

\begin{definition}\label{def:gen_SpecRad}
Let $\cE$ be an operator space and let $\cK$ be a Hilbert space. We define $\rho_\cE(X)$ for $X \in B(\cK) \otimes_{\min} \cE$,  to be the spectral radius of $X$, when considered as an element of $B(\cK) \otimes_{\min} C^*_{\max}(\cE)$. 
In other words
\[
\rho_\cE(X) = \rho(\iota(X)) = \lim_{k \to \infty}\|\iota(X)^k\|^{1/k}. 
\]
\end{definition}

We shall now describe a more concrete spectral radius. 
Suppose that $\cE \subseteq B(\cH)$ is a finite-dimensional operator space. Let $Q_1, \ldots, Q_d \in \cE$ be a basis, giving rise via \eqref{eq:identify_DQ} to an operator space structure on $\bC^d$, or, equivalently, to an nc operator ball $\bD_Q$. 
We define a joint spectral radius associated with $\bD_Q$ by making use of the Haagerup tensor product as follows (for background on the Haagerup tensor product, see for example \cite[Chapter 17]{PaulsenBook} or \cite[Section 1.5]{PisierBook}).

\begin{definition}\label{def:gen_SpecRad_concrete}
For $Q_1, \ldots, Q_d \in B(\cH)$ and $T = (T_1, \ldots, T_d) \in B(\cK)^d$, we define the {\em $Q$-spectral radius} of $T$ to be 
\begin{equation}\label{eq:Hsr}
\rho_Q(T) = \lim_{n \to \infty} \left \|\sum_{|w|=n} T^w \otimes Q_{w_1} \otimes_h Q_{w_2} \otimes_h \cdots \otimes_h Q_{w_n} \right\|^{1/n} 
\end{equation}
where the sum is over all words $w = w_1 w_2 \cdots w_n$ in $d$ letters $w_i \in \{1,2,\ldots, d\}$.
\end{definition}

In the above definition, we write $T^w$ for $T_{w_1} \cdots T_{w_n}$, and the elements
\[
\sum_{|w|=n} T_{w_1} \cdots T_{w_n} \otimes Q_{w_1} \otimes_h Q_{w_2} \otimes_h \cdots \otimes_h Q_{w_n}
\]
belong to the operator space
\[
B(\cK) \otimes_{\min} (\cE \otimes_h \cdots \otimes_h \cE),  
\]
where $\cE \otimes_h \cdots \otimes_h \cE$ is the $n$-fold Haagerup tensor product of $\cE$ with itself. 
As promised at the beginning of Section \ref{sec:pre}, a tensor product with no further indication is interpreted in the spatial (or minimal) sense. 
We use the decorated tensor products $\otimes_h$ to remind the reader that the element $Q_{w_1} \otimes_h Q_{w_2} \otimes_h \cdots \otimes_h Q_{w_n}$ belongs to $\cE \otimes_h \cdots \otimes_h \cE$. 
The norm appearing on the right hand side of \eqref{eq:Hsr} is the norm on $B(\cK) \otimes_{\min} (\cE \otimes_h \cdots \otimes_h \cE)$. 

Note that $\rho_Q$ depends on the choice of basis. 
We shall sometimes write $\rho_{\bD_Q}$ for $\rho_Q$. 
The two spectral radii we defined are closely related. If we consider $\cE$ to be $\bC^d$ with the operator space structure induced by $Q$ then $\rho_\cE(X) = \rho_Q(X)$ for every $X$. 

\begin{proposition}\label{prop:rho_Q_rho_E}
Let $Q_1,\ldots, Q_d \in B(\cH)$ and $T \in B(\cK)^d$. Let $\cE$ be the operator space spanned by $Q$, and let $X = \sum_{j=1}^d T_j \otimes Q_j \in B(\cK) \otimes_{\min} \cE$. Then, $\rho_Q(T) = \rho_\cE(X)$.   
\end{proposition}
\begin{proof}
This follows from the completely isometric identification of the $n$-fold Haagerup tensor product of $\cE$ with itself and the space of $n$-homogeneous polynomials in the algebra $OA_u(\cE) \subset C^*_{\max}(\cE)$; see Proposition 6.6 in \cite{PisierBook}. 
\end{proof}

In particular, the classical spectral radius formula implies that the limit in \eqref{eq:Hsr} exists. 
Examples will be given below. 

\begin{remark}
Let $Q_1,\ldots,Q_d \in B(\cH)$ and $T = (T_1,\ldots,T_d) \in B(\cK)^d$. We consider the unital operator algebra $\cA = \overline{\operatorname{Alg}(1, T_1,\ldots, T_d)}^{\|\cdot\|}$  generated by the $T_i$. If $\varphi \colon \cA \to B(\cL)$ is a unital completely contractive homomorphism, then we set $\varphi(T) = (\varphi(T_1), \ldots, \varphi(T_d))$. We have that
\[
\rho_Q(\varphi(T)) \leq \rho_Q(T).
\]
To see this note that $\varphi$ extends to a UCP map on $B(\cK)$. Therefore, for every word $w$, we have that $\varphi(T^w) = V^* \pi(T^w) V$, where $\pi \colon B(\cK) \to B(\cL)$ is a $*$-homomorphism and $V$ is an isometry. The inequality now follows. In particular, if $K_0 \subset K$ is a semi-invariant subspace for $\cA$, then
\[
\rho_Q(P_{K_0} T|_{K_0}) \leq \rho_Q(T).
\]
\end{remark}

Proposition \ref{prop:rho_Q_rho_E} implies the following lemma as both a generalization and an elucidation of \cite[Lemma 4.7]{SSS20}.

\begin{lemma}
For every holomorphic function $f \colon \bD \to M_n^d$, the function $\rho_Q(f(z))$ is log-subharmonic.
\end{lemma}
\begin{proof}
The spectral radius $\rho_Q(f(z))$ of the tuple valued function $f$ is equal, by definition, to the spectral radius of the operator valued function $\iota(f(z))$. Therefore, the log-subharmonicity result is an immediate consequence of \cite[Theorem 1]{Vesentini-maximum}, which says that the spectral radius is log-subharmonic.
\end{proof}

\begin{corollary}
Let $Q_1,\ldots,Q_d \in B(\cH)$ and let $f \colon \bD \to \bD_Q$ be an analytic function, such that $f(0) = 0$. Then, for every $z \in \bD$, $\rho_Q(f(z)) \leq |z|$. Moreover, $\rho_Q(f'(0)) \leq 1$.
\end{corollary}

We now come to our main result. Recall the definition of $\bD^{(\infty)}_Q$ from \eqref{eq:operator_ball}, \eqref{eq:DQinflevel} and \eqref{eq:DQinf}.
\begin{theorem}\label{thm:gen_SpecRad}
For every $d$-tuple of operators $T = (T_1, \ldots, T_d)$, 
\[
\rho_Q(T)<1 \Longleftrightarrow T \textrm{ is similar to a point in } \bD^{(\infty)}_Q.
\]
\end{theorem}

\begin{proof}
The backward implication is straightforward, and we dispose of it first. 
Since $\rho_Q$ is similarity invariant it suffices to show that $\rho_Q(T)<1$ for $T \in \bD^{(\infty)}_Q$. 
Now, $T \in \bD^{(\infty)}_Q$ means that $\|\sum T_j \otimes Q_j\|\leq t$ for some $t<1$. 
By one of the basic definitions of the Haagerup tensor product \cite[Equation (5.6)]{PisierBook}, we have that 
\[
\left\| \sum_{|w|=n} T_{w_1} \cdots T_{w_n} \otimes Q_{w_1} \otimes_h Q_{w_2} \otimes_h \cdots \otimes_h Q_{w_n}\right\| \leq t^n
\]
and it follows that $\rho_Q(T)\leq t <1$. 

Suppose now that $T \in B(\cK)^d$ and that $\rho_Q(T)<1$. 
We will show that $\pi_T \colon f \mapsto f(T)$ is a well-defined completely bounded homomorphism from $A(\bD_Q)$ into $B(\cK)$. 
Every $f \in A(\bD_Q)$ has a power series representation
\[
f(Z) = \sum_{n=0}^\infty \sum_{|w|=n} c_w Z^w = \sum_{n=0}^\infty f_n(Z) \in A(\bD_Q), 
\]
where, by Theorem 7.10 in \cite{KVV14}, this series converges uniformly and absolutely in $r\bD_Q$ for every $r\in (0,1)$. 
By \cite[Proposition 3.5]{SS25}, the norm of the homogeneous components $f_n$ satisfy $\|f_n\| \leq \|f\|$, and the Ces\`{a}ro means of this series also converge. 

Let $E = \cE^*$ be the operator space dual of $\cE$. 
By \cite[Proposition 6.6]{PisierBook} combined with the identification $A(\bD_Q) \cong OA_u(E)$ made in \cite[Section 7]{SS25}, we have that
\begin{equation}\label{eq:E_n_is_Haag}
E_n := \spn\{Z^w : |w|=n\} \subseteq A(\bD_Q)
\end{equation}
is completely isometric to the $n$-fold Haagerup tensor product $E \otimes_h \cdots \otimes_h E$ of $E$ with itself. 
Now by a fundamental property of the Haagerup tensor product of finite dimensional operator spaces \cite[Corollary 5.8]{PisierBook} 
\[
(\cE \otimes_h \cdots \otimes_h \cE)^* = \cE^* \otimes_h \cdots \otimes_h \cE^* = E \otimes_h \cdots \otimes_h E, 
\]
so that $E_n$ is the dual of $n$-fold Haagerup tensor product $\cE \otimes_h \cdots \otimes_h \cE$. 
For every $n$, the tuple $T$ defines an evaluation map $\pi_{n,T} \colon E_n \to B(\cK)$ given by 
\[
\pi_{n,T}(Z^w) = T^w  \,\, \FORAL w \textrm{ with } |w|=n .
\] 
By the assumption $\rho_Q(T)<1$, there is a $t \in (0,1)$ and a $n_0$ such that for all $n \geq n_0$, 
\be\label{eq:maint}
\left \|\sum_{|w|=n} T^w \otimes Q_{w_1} \otimes_h Q_{w_2} \otimes_h \cdots \otimes_h Q_{w_n} \right\| \leq t^n . 
\ee
We claim that \eqref{eq:maint} implies that 
\[
\|\pi_{n,T}\|_{cb} \leq t^n
\]
for all $n \geq n_0$. 
Indeed, recall that there is a completely isometric inclusion 
\[
B(\cK) \otimes_{\min} E^*_n \subseteq CB(E_n, B(\cK)), 
\] 
(see Equation (2.3.2) in \cite{PisierBook}, or \cite[Corollary 5.2]{blecher1991tensor}), in which the tensor 
\begin{equation}\label{eq:afore_mentioned}
\sum_{|w|=n} T^w \otimes Q_{w_1} \otimes_h Q_{w_2} \otimes_h \cdots \otimes_h Q_{w_n}
\end{equation}
is identified with $\pi_{n,T}$. 
To see why this is so, note that the set $\{Q_1, \ldots, Q_d\}$ corresponds to the basis $\{e_1, \ldots, e_d\}$ of $\cE$, which is a dual basis to the basis $\{Z_1, \ldots, Z_d\}$ of $E$. 
Therefore, $\{Q_{w_1} \otimes_h Q_{w_2} \otimes_h \cdots \otimes_h Q_{w_n}\}$ corresponds to the basis dual to $\{Z^w\}$. 
It follows from the identification of $\pi_{n,T}$ with the aforementioned element \eqref{eq:afore_mentioned} of $B(\cK) \otimes \cE \otimes_h \cdots \otimes_h \cE$, together with \eqref{eq:maint}, that $\|\pi_{n,T}\|_{cb} \leq t^n$ for $n \geq n_0$. Therefore, if $P_n \colon A(\bD_Q) \in E_n$ denotes the completely contractive map $f \mapsto f_n$ onto the $n$th homogeneous component of $f$, then the series of maps 
\[
\sum_{n=0}^\infty \pi_{n,T} \circ P_n \colon f \mapsto \sum_{n=0}^\infty \pi_{n,T}(f_n) = \sum_{n=0}^\infty f_n(T)
\]
converges in the completely bounded norm to the evaluation representation $\pi_T \colon f \mapsto f(T)$, which must be completely bounded. In particular, for every $f \in A(\bD_Q)$ the series $\sum_{n=0}^\infty f_n(T)$ converges in norm. 

Now, recall that every completely bounded representation is similar to a completely contractive one \cite[Theorem 9.1]{PaulsenBook}. Thus $\pi_T$ is similar to a completely contractive representation. 
By choosing a sufficiently small $\epsilon$ such that $\rho_\cE((1+\epsilon)T)<1$ and repeating the above argument with $(1+\epsilon)T$ instead of $T$, we obtain that $\pi_{(1+\epsilon)T}$ is similar to a completely contractive evaluation representation, and one sees that this completely contractive representation must be an evaluation $\pi_Y$ for $Y \in B(\cK)^d$ that is jointly similar to $(1+ \epsilon)T$. 
It follows that $\pi_{1,Y}$ is completely contractive, or in other words $\|\sum_j Y_j \otimes Q_j\| \leq 1$. 
We conclude that $T$ is similar to the point $(1+\epsilon)^{-1}Y  \in \bD^{(\infty)}_Q$. 
\end{proof}

\begin{remark}
One can easily show that for a reducible point $X \in \bM^d$ with Holder-Jordan parts $X^{(1)}$, $X^{(2)}, \ldots$, $X^{(m)}$, 
\[
\rho_Q(X) = \max\left\{ \rho_Q(X^{(1)}), \rho_Q(X^{(2)}), \ldots, \rho_Q(X^{(m)}) \right\}.
\]
In particular, if $X$ is a commuting tuple, then the Holder-Jordan parts $X^{(1)}$, $X^{(2)}, \ldots$, $X^{(m)}$ are points in $\bC^d$ that constitute the joint spectrum $\sigma(X)$ of $X_1, \ldots, X_d$, and in this case we see that $\rho_Q(X)<1$ if and only if $\sigma(X) \subseteq \bD_Q(1)$. 
We conclude that the condition for a commuting tuple of matrices to be jointly similar to a point in $\bD_Q$ depends only on the first level $\bD_Q(1)$, in other words: it depends only on the norm structure --- and not on the operator space structure --- of the corresponding operator space $\cE$. 
\end{remark}

We now consider a related spectral radius-like quantity, which might motivate the above-defined spectral radius, and then we compare the two in a couple of examples. 
\begin{definition}\label{def:min_spectral_rad}
For $Q_1, \ldots, Q_d \in B(\cH)^d$ and $T \in B(\cK)^d$, we define the {\em minimal $Q$-spectral radius} of $T$ to be
\begin{equation}\label{eq:spectral_radius_formula}
\rho^{\min}_Q(T) = \lim_{n \to \infty} \left \|\sum_{|w|=n} T_{w_1} \cdots T_{w_n} \otimes Q_{w_1} \otimes Q_{w_2} \otimes \cdots \otimes Q_{w_n} \right\|^{1/n} . 
\end{equation}
\end{definition}

Note that the tensor product in the formula above is the minimal (spatial) tensor product. 
A familiar log-subadditivity argument can be used to show that the above limit exists. 
We shall sometimes write $\rho^{\min}_{\bD_Q}$ for $\rho^{\min}_Q$. 
Note that one always has $\rho^{\min}_Q(T) \leq \rho_Q(T)$ because the Haagerup tensor norm dominates the minimal one.

\begin{example}\label{ex:row_ball_min}
Let $Q$ be as in Example \ref{ex:row_ball}. 
The sum in \eqref{eq:spectral_radius_formula} amounts to a long block operator row vector indexed by all words $w$ of length $n$, where at the slot indexed by $w$ we have the power $T^w = T_{w_1} \cdots T_{w_n}$. 
The norm of this block operator row vector is given by $\| \sum_{|w| = n} T^{w} (T^{w})^* \|$, and so the minimal spectral radius is given by formula \eqref{eq:Pop_spec_rad}. 
Thus, in this case, we recover the notion of joint spectral radius used by Bunce, Popescu, and others. 
Since the Haagerup tensor product of row spaces is isometrically isomorphic to the minimal tensor product of row spaces (see Equation (5.16) in \cite{PisierBook}), we conclude that 
\[
\rho^{\min}_{\fB_d}(T) = \rho_{\fB_d}(T) .
\]
We see that in the case of the row ball, the spectral radius of Definition \ref{def:gen_SpecRad_concrete} coincides with the one described in \eqref{eq:Pop_spec_rad}. 
\end{example}

\begin{example}\label{ex:nc_polydisk_min}
Let $Q$ be as in Example \ref{ex:nc_polydisc}. 
Now the sum in \eqref{eq:spectral_radius_formula} amounts to a large block diagonal operator matrix indexed by words $w$ of length $n$, where at the diagonal slot indexed by $w$ we have the power $T^w = T_{w_1} \cdots T_{w_n}$. 
The norm of this operator is
\[
\max_{|w|=n} \|T^w\|. 
\]
Thus, in this case, $\rho^{\min}_Q(T)$ is equal to the joint spectral radius introduced by Rota and Strang \cite{RS60}, defined by 
\[
\rho_{RS}(T) := \lim_{n \to \infty} \max_{|w| = n} \|T^{w}\|^{1/n}.
\]
The reader is warned that $\rho_{RS}(T)$ is also referred to as ``the joint spectral radius" in parts of the literature. 
In \cite[Proposition 1]{RS60}, Rota and Strang prove that $\rho_{RS}(T) \leq 1$ if and only if there exists an equivalent norm {\em on the algebra generated by $T$} such that every $T_i$ is a contraction. 
However, in \cite{ando1998simultaneous} $d$-tuples of matrices $T$ are exhibited such that $\rho_{RS}(T)<1$ but $T$ is not jointly similar to a tuple of strict contractions. 
By Theorem \ref{thm:gen_SpecRad}, this means that
\[
\rho_{RS}(T) = \rho^{\min}_{\fD_d}(T) \neq \rho_{\fD_d}(T). 
\]
So, in general, Definition \ref{def:gen_SpecRad_concrete} is different from Definition \ref{def:min_spectral_rad}.  
\end{example}

Theorem \ref{thm:gen_SpecRad} has the following immediate corollary, which answers the main question that we raised in this paper. 

\begin{corollary}\label{cor:main_result_abstract}
Let $\cE$ be a finite-dimensional operator space, and let $X \in M_n(\cE)$. 
Then $\rho_\cE(X) < 1$ if and only if there exists $S \in \GL_n $ such that $\|S^{-1}XS\|<1$.
Consequently, 
\[
\rho_\cE(X) = \inf\left\{\|S^{-1}XS\| : S \in \GL_n \right\} .
\]
Moreover, if $X$ is irreducible, then the infimum is attained.
\end{corollary}
\begin{proof}
The first claim follows directly from Theorem \ref{thm:gen_SpecRad} and Proposition \ref{prop:rho_Q_rho_E}. 
The equality then follows from a scaling argument: if $\rho_\cE(X)=r$, consider $X_\epsilon = (r+\epsilon)^{-1}X$. 
Then $\rho_\cE(X_\epsilon)<1$, so there exists $S$ such that $\|S^{-1}X_\epsilon S\|<1$, whence $\|S^{-1}X S\|<r+\epsilon$. 
On the other hand, if $S$ is invertible then $\rho_\cE(X) = \rho_\cE(S^{-1}XS) \leq \|S^{-1} XS\|$.

Now fix an irreducible $X \in M_n^d$. Then, $\rho_{\cE}(X) \leq \|X\|$. By the Artin-Voigt lemma \cite[Corollary 1.3]{Gabriel-fin_rep_type}, the similarity orbit of $X$ is an algebraic variety. Hence, in particular, the similarity orbit of $X$ intersected with the closed $\cE$-ball in $M_n^d$ of radius $\|X\|$ is compact. Hence, the norm attains its minimum on the orbit.
\end{proof}

As the example in Remark \ref{rem:matrix} shows, the infimum need not be attained if $X$ is reducible.

\section{Application to noncommutative rational functions}\label{sec:nc_rat}

\subsection{Background on noncommutative rational functions}

We recall the basics of the theory of nc rational functions. Nc rational functions are elements of the free skew field. The latter is the universal skew field of fractions of the free algebra defined by Cohn and Amitsur. In our description, we will follow \cite{KVV-sing_nc_rational} and \cite{Volcic-rat}. An nc rational expression in $d$ nc variables is any syntactically valid expression that one can write using addition, multiplication, and inverse in nc polynomials in $d$ nc variables. For example, if we use two nc variables:
\[
(xy - yx)^{-1},\, xy^{-1}x^{-1} y + 7(1-xyx)^{-1},\, (1 - x y y^{-1} x^{-1})^{-1}.
\]
We say that an expression is admissible if there is a $d$-tuple of matrices on which one can evaluate the expression. Namely, every expression that we invert is invertible for this $d$-tuple. Clearly, the first two expressions in our example are admissible, but the last one is not. For an admissible expression, we define its domain to be the set of $d$-tuples of matrices of all sizes on which we can evaluate the expression. Plugging in matrices of commuting variables (generic matrices), we can see that if the domain of an admissible expression is non-empty on level $n$, then it is a Zariski open dense set. In particular, the domains of every pair of admissible nc rational expressions intersect. We say that two nc rational expressions are equivalent if they evaluate to the same matrix on every $d$-tuple of matrices in the intersection of their domains. Finally, an nc rational function is an equivalence class of admissible nc rational expressions. The domain of an nc rational function is the set of all $d$-tuples on which we can evaluate at least one of the relations in our equivalence class. We will denote the domain of an nc rational function $f$  by $\dom(f)$.

This construction can be simplified if we assume that our nc rational function has $0$ in its domain. This will be our assumption from now on. Let $f$ be an nc rational function with $0 \in \dom(f)$. Then, there exist matrices $A_1,\ldots, A_d \in M_n(\C)$ and vectors $b, c \in \C^n$, such that for all points $X \in \bM^d$ for which the pencil $I - \sum_{j=1}^d X_j \otimes A_j$ is invertible, we have
\[
f(X) = (I \otimes c)^* \left(I - \sum_{j=1}^d X_j \otimes A_j \right)^{-1} (I \otimes b).
\]
This is called a (descriptor) realization of $f$. We will say that $(A,b,c)$ is a realization of $f$. It is clear that if we replace the tuple $(A_1,\ldots,A_d)$ by a similar tuple $(S^{-1} A_1 S,\ldots, S^{-1} A_d S)$ and the vectors by $S^{-1} b$ and $S^* c$, respectively, then the above formula will not change. By \cite{KVV-sing_nc_rational} and \cite{Volcic-rat}, there exists a tuple of minimal size, such that $(A,b,c)$ is a realization of $f$. All such minimal realizations are similar. Moreover, if $(A,b,c)$ is a minimal realization, then
\begin{equation}\label{eq:domf}
\dom(f) = \left\{X \in \bM^d \mid \det \left(I - \sum_{j=1}^d X_j \otimes A_j \right) \neq 0 \right\}.
\end{equation}
There is also a slightly different form of a realization of an nc rational function, called a Fornasini-Marchesini (FM for short) realization, which has the form
\[
f(X) = d I + (c \otimes I)^* \left(I - \sum_{j=1}^d X_j \otimes A_j \right)^{-1} \left( \sum_{j=1}^d X_j \otimes b_j \right).
\]
In many cases, the FM realization is more convenient for function-theoretic purposes. An equivalent result for domains of nc rational functions holds for the FM realization \cite[Lemma 2.1]{MartinS-rational_cuntz}. The same realization theory applies to matrix-valued rational nc functions.

We shall require the following lemma. 

\begin{lemma} \label{lem:bounded_inverse_pencil}
Let $A_0, A_1, \ldots, A_d \in M_n$. Let $A_0 + \sum_{j=1}^d z_j A_j$ be a linear pencil. Assume that $\det\left( I \otimes A_0 + \sum_{j=1}^d X_j \otimes A_j\right) \neq 0$, for every $X \in \overline{\bD_Q}$. Then the pencil is invertible in $A(\bD_Q)$. 
\end{lemma}
\begin{proof}
Let $g(X) = (I \otimes A_0 + \sum_{j=1}^d X_j \otimes A_j)^{-1}$. Clearly, $g$ is an $M_n$-valued nc function as a composition of nc functions. We first show that $g$ is bounded on $\overline{\bD_Q}$. If this is not the case, then there is a sequence $X^k \in \overline{\bD_Q(m_k)}$ such that $\|g(X^k)\| \to \infty$. 
We show this is impossible. 

Suppose first that the dimensions $m_k$ are bounded, say $m_k \leq N$ for all $k$; 
we may assume in fact that $m_k = N$ for all $k$. 
But then by the compactness of the $N$-th level $\ol{\bD_Q(N)}$ we may, by passing to a subsequence, assume that $X^k \to X^0 \in \ol{\bD_Q(N)}$. 
But by assumption, $I \otimes A_0 + \sum_{j=1}^d X^0_j \otimes A_j$ is invertible, and by continuity of the inverse $g(X^k) \to g(X^0)$, contradicting the assumption that $\|g(X^k)\| \to \infty$.

We now show that if $m_k \to \infty$, then one can find $N \in \bN$ and $Y^k \in \overline{\bD_Q(N)}$ such that $\|g(Y^k)\| \to \infty$. 
This will contradict what we demonstrated in the previous paragraph and will complete the proof that $g$ is bounded on $\overline{\bD_Q}$. 

Now, $\|g(X^k)\| \to \infty$ if and only if there is a sequence $v_k \in \bC^{m_k} \otimes \bC^n$ of unit vectors such that 
\[
\left\|(I \otimes A_0 + \sum_{j=1}^d X^k_j \otimes A_j)v_k \right\| \to 0 .
\]
Letting $e_1, \ldots, e_n$ denote the standard basis of $\bC^n$, we can write $v_k$ as $v_k = \sum_{\ell=1}^n u^k_\ell \otimes e_\ell$, for some vectors $u^k_\ell \in \bC^{m_k}$. 
Let $P_k$ be the orthogonal projection of $\bC^{m_k}$ onto the subspace
\[
V_k = \spn\left(\left\{u^k_\ell : \ell=1,\ldots, n \right\} \bigcup \left\{ X_j^k u^k_\ell : \ell=1,\ldots, n ; j = 1, \ldots, d\right\} \right). 
\]
Clearly, $\dim(V_k) \leq N:= n \times (d+1)$. 
We may as well assume that $\dim(V_k) = N$. 
We note that $w_k = (P_k \otimes I_n) v_k = v_k$ is still a unit vector, and if we let $Y^k = P_k X^k P_k$, then $Y^k$ can be considered as a tuple in $\ol{\bD_Q(N)}$ and 
\[
\left\|(I \otimes A_0 + \sum_{j=1}^d Y^k_j \otimes A_j)w_k \right\| = \left\|(I \otimes A_0 + \sum_{j=1}^d X^k_j \otimes A_j)v_k \right\| \to 0 .
\]
So, we have found a sequence $Y^k \in \ol{\bD_Q(N)}$ such that $\|g(Y^k)\| \to \infty$, which is impossible. 

To see that $g$ is uniformly continuous in $\ol{\bD_Q}$, we use the resolvent identity. Namely, for $X, Y \in \bD_Q(m)$, 
\[
\left\|g(X) - g(Y)\right\| = \left\|g(X) \left( \sum_{j=1}^d (Y_j - X_j) \otimes A_j \right) g(Y) \right\|
\]
Let $E$ be the operator space structure on $\C^d$, such that $\overline{\bD_Q}$ is the closed operator unit ball of $E$. Then the linear map $\varphi_A(X) = \sum_{j=1}^d X_j \otimes A_j$ is a completely bounded map from $E$ to $M_n$. Since $g$ is bounded on $\overline{\bD_Q}$, we get that
\[
\left\|g(X) - g(Y) \right\| \leq \|g\|_{\infty, \overline{\bD_Q}}^2 \|\varphi_A\|_{cb} \|X - Y\|_E
\]
We conclude that $g$ is uniformly Lipschitz on $\overline{\bD_Q}$ and, thus, uniformly continuous.
\end{proof}

The following immediate corollary of the above lemma shows that a rational function that is defined on an nc ball $\bD_Q$ is bounded on every nc ball $r\bD_Q$ of radius $r<1$. 
It is significant and somewhat surprising that this statement is false for general nc functions (see \cite{Pascoe20}).

\begin{corollary}\label{cor:bounded_rational}
Let $\bD_Q \subset \bM^d$ be an nc operator ball, let $A \in M_n^d$, and suppose that 
\[
\bD_Q \subset \left\{X \in \bM^d \mid \det \left(I - \sum_{j=1}^d X_j \otimes A_j \right) \neq 0 \right\}. 
\]
Then $g(X) = (I - \sum X_j \otimes A_j)^{-1}$ is a $M_n$-valued nc function on $\bD_Q$ which is bounded on $r\bD_Q$ for all $r<1$. 
Consequently, $g$ is uniformly continuous on $r\bD_Q$ for all $r<1$. 
\end{corollary}
\begin{proof}
For every $0 < r < 1$, the pencil $I - \sum X_j \otimes A_j$ is invertible on $r \overline{\bD_Q} = \overline{\bD_{rQ}}$. Thus, the pencil is invertible in $A(r \bD_Q)$ by the preceding lemma.
\end{proof}

\begin{corollary} \label{cor:rational_spectrum}
Let $f \in M_k(A(\bD_Q))$ be an nc rational function. Then,
\[
\sigma_{A(\bD_Q)}(f) = \bigcup_{X \in \ol{\bD_Q}} \sigma(f(X)).
\]
\end{corollary}
\begin{proof}
Let $f \in M_k(A(\bD_Q))$ be nc rational. It suffices to show that $f$ is invertible if and only if $f(X)$ is invertible for every $X \in \overline{\bD_Q}$. One direction is trivial, so assume $f(X)$ is invertible for every $X \in \overline{\bD_Q}$. In particular, $f(0) \neq 0$. Hence, by \cite{KVV-sing_nc_rational}, if $f$ has minimal FM realization
\[
f(Z) = f(0) \otimes I + (C^* \otimes I) \left( I - \sum_{j=1}^d A_j \otimes Z_j\right)^{-1} \left( \sum_{j=1}^d B_j \otimes Z_j \right),
\]
then $f^{-1}$ has a minimal FM realization with pencil $I - \sum_{j=1}^d A_j^{\times} \otimes Z_j$. Here $A_j^{\times}$ is a rank one perturbation of the $A_j$. However, since $f^{-1}(X)$ is defined for every $X \in \overline{\bD_Q}$, we have that the pencil is invertible at every point of $\overline{\bD_Q}$. By Lemma \ref{lem:bounded_inverse_pencil}, we have that $f^{-1} \in M_k(A(\bD_Q))$ and we are done.
\end{proof}

\subsection{The spectral radius and the domain of nc rational functions}
Let $\cE = \bC^d$ equipped with an operator space structure, and let $E = \cE^*$ denote its dual. The following is a partial extension of \cite[Theorem A]{JMS-rat}.

\begin{theorem} \label{thm:bdd_rat_func}
Let $f$ be an nc rational function with $0$ in its domain. Let $(A,b,c)$ be a minimal realization of $f$, and put $r=\rho_{\bB_E}(A)$. Then $(1/r)\bB_\cE \subset \dom(f)$. Conversely, if $R\bB_\cE \subset \dom(f)$, then $r\leq 1/R$. 
In particular, $\rho_{\bB_E}(A)<1$ if and only if $\dom(f)$ contains $R \bB_\cE$ for some $R > 1$, in which case $f \in A(\bB_{\cE})$. 
\end{theorem}
\begin{proof}
Set $r = \rho_{\bB_E}(A)$, and let $\epsilon > 0$. 
By Corollary \ref{cor:main_result_abstract}, there exists $S \in \GL_n$, such that $\|S^{-1} A S\|_E < r + \epsilon$. 
Replacing $A$ by $S^{-1} A S$, $b$ by $S b$, and $c$ by $S^{-1 *} c$, we can assume that we have a minimal realization with $\|A\|_E < r + \epsilon$. Let $X \in (r+\epsilon)^{-1} \bB_{\cE}$. Denote by $\varphi_A$ the map on $\cE$ corresponding to $A \in M_m(E) = M_m(\cE^*)$. Since $\|\varphi_A\|_{cb} = \|A\|_E < r + \epsilon$, we have that
\[
\left\| \sum_{j=1}^d A_j \otimes X_j \right\| = \|\varphi_A^{(m)}(X)\| \leq \|A\|_E \|X\|_\cE < 1.
\]
Since $\sum_{j=1}^d A_j \otimes X_j$ is unitarily equivalent to $\sum_{j=1}^d X_j \otimes A_j$, $\det(I - \sum_{j=1}^d X_j \otimes A_j) \neq 0$. 
It then follows from \eqref{eq:domf} that
$X \in \dom(f)$. Since $X$ was an arbitrary point in $(r+\epsilon)^{-1} \bB_{\cE}$ and $\epsilon$ was arbitrarily small, we have $(1/r)\bB_{\cE} \subset \dom(f)$.

Conversely, suppose that $R\bB_\cE \subset \dom(f)$ for some $R>0$. Then, by Corollary \ref{cor:bounded_rational}, for any $t<R$, the function $g(Z) = \left(I - \sum_{j=1}^d Z_j \otimes A_j \right)^{-1}$ is bounded on $t \bB_\cE$. 
It follows that for all $t<R$, the function $g_t(Z):= g(tZ)$ is bounded on $\bB_\cE$, and by Proposition \ref{prop:series} it has a convergent power series there.
Since the power series is equal to the Neumann series
\[
g_t(Z) = \sum_{m=0}^\infty \sum_{|w|=m} t^m Z^w \otimes A^w 
\]
near $0$, this is the power series in $\bB_\cE$. 
By Proposition \ref{prop:series}, the homogeneous terms are bounded: 
\[
\|\sum_{|w|=m} t^m A^w \otimes Z^w \| \leq \|g_t\|\,\,, \,\, \textrm{ for all } m = 0,1,2,\ldots
\]
whence $\limsup \|\sum_{|w|=m}  A^w \otimes Z^w \|^{1/m} \leq t^{-1}$.  
But by equations \eqref{eq:Hsr} and \eqref{eq:E_n_is_Haag}, 
\[
\rho_E(A) = \lim_{m\to \infty} \|\sum_{|w|=m} A^w \otimes Z^w\|^{1/m}, 
\]
so we conclude that $\rho_E(A) \leq t^{-1}$. 
Since this is true for all $t<R$, we obtain $\rho_E(A) \leq 1/R$. 

Here is an alternative proof of the converse direction, given without loss for the case $R=1$.
Recall  that $\rho_{\bB_E}(A) > 1$ is equivalent to $\rho\left(\sum_{j=1}^d A_j \otimes Z_j \right) > 1$ in $M_n \otimes A(\bB_{\cE})$. Let $\lambda \in \sigma\left( \sum_{j=1}^d A_j \otimes Z_j \right)$ be such that $|\lambda| > 1$. By Corollary \ref{cor:rational_spectrum}, there exists $X \in \ol{\bB_{\cE}}$, such that $\lambda \in \sigma\left(\sum_{j=1}^d A_j \otimes X_j\right)$. Multiplying $X$ by a unimodular scalar and $1/|\lambda|$, we get a new point $X' \in \bB_{\cE}$, such that $1 \in \sigma\left(\sum_{j=1}^d A_j \otimes X_j'\right)$. Therefore, $X' \notin \dom(f)$.
\end{proof}

\begin{remark}
The above theorem holds with the same proof if $f$ is taken to be an admissible rational expression with a not necessarily minimal descriptor representation $(A,b,c)$, and the domain is understood to be the domain of the expression. 
\end{remark}
\begin{remark}
In practice, we might want to use the above theorem for some concrete nc operator unit ball $\bD_Q$. 
In that case, the dual spectral radius one should use is $\rho_{\bD_Q^\circ}$, where the {\em polar dual} of $\bD_Q$ is defined as 
\[
\bD_Q^\circ = \left\{Y \in \bM^d : \|\sum_j Y_j \otimes X_j\| < 1 \textrm{ for all } X \in \bD_Q  \right\}.
\]
\end{remark}

\section{A case study: the famous example}

In \cite[Theorem A]{JMS-rat}, a version of Theorem \ref{thm:bdd_rat_func} was obtained for the case in which $\cE$ is the row operator space. 
Note that the only spectral radius considered in \cite{JMS-rat} was the one described by \eqref{eq:Pop_spec_rad}.
However, in that special case the spectral radii $\rho_{\bB_E} = \rho_{\fC_d}$ and $\rho_{\bB_\cE} = \rho_{\fB_d}$ corresponding to the column ball $\fC_d$ and the row ball $\fB_d$ coincide with \eqref{eq:Pop_spec_rad}, since the Haagerup tensor product of row/column operator spaces is the same as the minimal tensor product, which is again a row/column operator space (see Example \ref{ex:row_ball_min}). 
The result in \cite{JMS-rat} is sharper, in that it shows that $\rho_{\fB_d}(A) = \rho_{\fC_d}(A)<1$ is a necessary condition for $f$ to be bounded on $\fB_d$. 
The following example shows that this implication does not hold for all nc operator balls. 

\subsection{An nc rational function on the nc bidisc}\label{subsec:case}
The ``famous example'' is the rational function $f(z,w) = \frac{2zw - z - w}{2 - z - w}$, which is bounded on the bidisc. However, it is singular on the boundary at the point $(1,1)$. Using Agler's approach, one can obtain a realization for $f$ that extends to a bounded nc rational function on $\fD_2$. Here is a realization of an nc lift of $f$:
\[
f(Z,W) = \begin{pmatrix} 0 & 0 & 1 \end{pmatrix} \left( I - \underbrace{\begin{pmatrix} 1/2 & 0 & 0 \\ -1/2 & 0 & 0 \\ 1/\sqrt{2} & 0 & 0 \end{pmatrix}}_{A_1} Z - \underbrace{\begin{pmatrix} 0 & -1/2 & 0 \\ 0 & 1/2 & 0 \\ 0 & 1/\sqrt{2} & 0 \end{pmatrix}}_{A_2} W \right)^{-1} \begin{pmatrix} -1/\sqrt{2} \\ -1/\sqrt{2} \\ 0 \end{pmatrix}.
\]
It is not hard to check that this realization is minimal. 

In fact, $f$ defines a bounded nc function on $\fD_2$ with supremum norm equal to $1$. 
To see this, a straightforward Schur complement argument shows that $f$ also has an FM realization
\be\label{eq:realization}
f(Z,W) = -\left[ \begin{matrix} 1/\sqrt{2} & 1/\sqrt{2} \end{matrix}\right] \left[ \begin{matrix} Z & 0 \\0 & W \end{matrix} \right]
\left(I - \left[ \begin{matrix} 1/2 & -1/2 \\-1/2 & 1/2 \end{matrix} \right] \left[ \begin{matrix} Z & 0 \\0 & W \end{matrix} \right]\right)^{-1} \left[ \begin{matrix} 1/\sqrt{2} \\ 1/\sqrt{2} \end{matrix} \right]. 
\ee
Since, with 
\[
B = \left[ \begin{matrix} 1/\sqrt{2} & 1/\sqrt{2} \end{matrix} \right], 
C = \left[ \begin{matrix} 1/\sqrt{2} \\ 1/\sqrt{2} \end{matrix} \right], 
D = \left[ \begin{matrix} 1/2 & -1/2 \\-1/2 & 1/2 \end{matrix} \right], 
\]
we have that 
\be\label{eq:V}
V = \left[\begin{matrix}0 & B \\ C & D  \end{matrix}\right] \colon \bC \oplus \bC^2 \to \bC \oplus \bC^2
\ee
is a unitary matrix, we conclude from \cite[Theorem 8.1]{agler2015} that $f$ is bounded with $\|f\|_\infty \leq 1$. 

Since the largest $R$ for which $R\fD_2 \subset \dom(f)$ is $R=1$, we get from Theorem \ref{thm:bdd_rat_func} that $\rho_{\fD^\circ_2}(A) = 1$, where $\fD_2^\circ$ is the concrete nc operator unit ball we use to represent $(\min(\ell^\infty_d))^* = \max(\ell^1_d)$. 
Let us verify directly from the definition that $\rho_{\fD^\circ_2}(A) = 1$. 
Note that by Theorem \ref{thm:bdd_rat_func} this gives an alternative justification that $\dom(f)$ contains $\fD_2$; however, the theorem cannot be used to show that $f$ is bounded on $\fD_2$. 

Let $U_1$ and $U_2$ be the generators of $C^*_u(\bF_2)$.
By Theorem \cite[Corollary 8.13]{PisierBook}, we know that homogeneous polynomials in $U_1,U_2$ of degree $n$ span an operator space that is completely isometric to the $n$-fold Haagerup tensor product of $\max(\ell^1_d)$ with itself. 
Therefore, 
\[
\rho_{\fD_2^{\circ}}(A) = \lim_{n\to \infty} \left\| \sum_{|w|=n} A^{w} \otimes U^{w} \right\|^{1/n}.
\]
We write $A_1 = v_1 e_1^*$ and $A_2 = v_2 e_2^*$, where $\{e_1, e_2, e_3\}$ is the standard basis for $\bC^3$ and 
\[
v_1 = \begin{pmatrix} 1/2 \\ -1/2 \\ 1/\sqrt{2} \end{pmatrix} \text{ and } v_2 = \begin{pmatrix} -1/2 \\ 1/2 \\ 1/\sqrt{2} \end{pmatrix}.
\]
Then, $A_1^2 = \frac{1}{2} A_1$, $A_1 A_2 = -\frac{1}{2} v_1 e_2^*$, $A_2 A_1 = -\frac{1}{2} v_2 e_1^*$, and $A_2^2 = \frac{1}{2} A_2$. More generally, it follows by induction that for every non-empty word $w = w_1 \cdots w_n$ on the alphabet $\{1,2\}$, we have
\be\label{eq:Aalpha1}
A^{w} = \frac{1}{2^{n-1}}B_w, 
\ee
where 
\be\label{eq:Aalpha2}
B_w := (-1)^{m(w)}v_{w_1} e_{w_n}^* \quad \textrm{ and } \quad m(w) = \left| \{ i \mid w_i \neq w_{i+1}\}\right|.
\ee
Therefore 
\[
\left\| \sum_{|w|=n} A^{w} \otimes U^{w} \right\|^{1/n}
\leq \left( \frac{1}{2^{n-1}}\sum_{|w|=n} \|B_w \otimes U^{w}\| \right)^{1/n}
\leq \left(\frac{2^n}{2^{n-1}}\right)^{1/n} \to 1, 
\]
and this shows that $\rho_{\fD^\circ_2}(A) \leq 1$. 

To obtain the reverse inequality, we observe that by the universality of the $U$'s, 
\[
\left\|\sum_{|w|=n} A^{w} \otimes U^{w} \right \|\geq \left\|\sum_{|w|=n} A^{w}\right \|
\]
We now use \eqref{eq:Aalpha1} and \eqref{eq:Aalpha2}, noting that
\[
(-1)^{m(w)} = \begin{cases} 
1 & w_1 = w_n \\
-1 & w_1 \neq w_n 
\end{cases}.
\]
So in the sum $\sum_{|w|=n} A^{w}$ there are only four rank one matrices adding up, each matrix appears exactly $2^{n-2}$ times, and the signs work out in such a way that we obtain
\[
\left\|\sum_{|w|=n} A^{w}\right \| = \frac{1}{2^{n-1}}\left\|\sum_{|w|=n} B_w \right\| = \frac{1}{2^{n-1}} \left\|2^{n-2} \begin{pmatrix} 1 & -1 & 0 \\
-1 & 1 & 0 \\ 
0 & 0 & 0 \end{pmatrix} \right\| = 1. 
\]
It follows that 
\[
\left\| \sum_{|w|=n} A^{w} \otimes U^{w} \right\|^{1/n} 
\geq 1
\]
which implies that $\rho_{\fD_2^\circ}(A) \geq 1$. 
This concludes the direct verification that $\rho_{\fD_2^\circ}(A)=1$.

\subsection{Switching the roles of the operator space and its dual}\label{subsec:switch}
Switching the roles of $\fD_2$ and $\fD_2^\circ$ we now turn to the computation of $\rho_{\fD_2}(A)$. 
By Theorem \ref{thm:bdd_rat_func}, this quantity will determine scales of the nc diamond $\fD_2^\circ$ in which $f$ defines a rational function. 
By \eqref{eq:Aalpha1} and \eqref{eq:Aalpha2}, it is easy to compute
\[
\rho_{RS}(A) = \lim_{n\to \infty} \max_{|w|=n} \|A^{w}\|^{\frac{1}{n}} = \lim_{n\to\infty} 2^{-\frac{n-1}{n}} = \frac{1}{2}.
\]
However, we have seen in Example \ref{ex:nc_polydisk_min} that $\rho_{RS}$ is in general not equal to $\rho_{\fD_2}$. 
It turns out that $\rho_{\fD_2}(A) = \frac{1}{2}$ too, but we will need to work harder to establish this. 

Let $n \in \bN$ and let $\sigma_1,\ldots,\sigma_n \colon \min(\ell^{\infty}_2) \to B(\cH)$ be completely contractive maps. Write $\{f_1,f_2\}$ for the standard basis vectors in $\min(\ell^{\infty}_2)$. For $n \geq 1$ and $j \in {1,2}$, set
\[
T_{j,n} = \sum_{\substack{|w|=n \\ w_n = j}} \sigma_1(f_{w_1}) \sigma_2(f_{w_2}) \cdots \sigma_n(f_{w_n}).
\]
and for $n \geq 2$ set
\[
T_{i,j,n} = \sum_{\substack{|w|=n \\ w_1 = i, w_n = j}} \sigma_1(f_{w_1}) \sigma_2(f_{w_2}) \cdots \sigma_n(f_{w_n}).
\]

\begin{lemma} \label{lem:norm_estimate_T}
With the above notation, for every $n \geq 2$
and every $j \in \{1,2\}$, 
it holds that $\|T_{j,n}\| \leq 1$ and $\|T_{1,j,n}\pm T_{2,j,n}\|\leq 1$. 
\end{lemma}
\begin{proof}
For $n = 1$, we have that 
$\|T_{j,1}\| = \|\sigma(f_j)\| \leq 1$. 
If $n\geq 2$ then 
\begin{align*}
\|T_{1,j,n} \pm T_{2,j,n} \| 
&= \|\sigma_1(f_1) T_{j,n-1} \pm \sigma_1(f_2) T_{j,n-1} \| \\
&= \| \sigma_1(f_1 \pm f_2) T_{j,n-1} \| \leq \|T_{j,n-1}\|,
\end{align*}
because $\|f_1\pm f_2\| \leq 1$. On the other hand, if $n \geq 2$, then $\|T_{j,n}\| = \|T_{1,j,n} + T_{2,j,n}\| \leq \|T_{j,n-1}\|$. 
The proof is finished by induction on $n$.
\end{proof}

\begin{proposition}
Let $A = (A_1,A_2)$ be the tuple in the realization of the ``famous example" $f$, then $\rho_{\fD_2}(A) = 1/2$.
Consequently, $2 \fD_2^\circ \subset \dom(f)$. 
\end{proposition}
\begin{proof}
We already know that $\rho^{\min}_{\fD_2}(A) = \rho_{RS}(A) = 1/2$, therefore $\rho_{\fD_2}(A) \geq 1/2$, because the Haagerup tensor norm dominates the minimal one. 
We proceed to obtain the reverse inequality. 
To compute $\rho_{\fD_2}(A)$ we need to estimate the norms
\[
\left\|\sum_{|w| = n} A^w \otimes f_{w_1} \otimes_h \cdots \otimes_h f_{w_n}\right\|
\]
By an equivalent characterization of the Haagerup tensor norm \cite[Corollary 5.3]{PisierBook}, this norm is the supremum of 
\[
\left\|\sum_{|w| = n} A^w \otimes \sigma_1(f_{w_1}) \cdots \sigma_n(f_{w_n})\right\|
\]
where the supremum is over all completely contractive maps $\sigma_1,\ldots,\sigma_n \colon \min(\ell^{\infty}_2) \to B(\cH)$. 
Now, $A^w \otimes \sigma_1(f_{w_1}) \cdots \sigma_n(f_{w_n})$ is a $3 \times 3$ block operator matrix, and by \eqref{eq:Aalpha1} and \eqref{eq:Aalpha2} we can compute all its entries. 
For example, the $(1,1)$ entry is 
\[
\frac{1}{2^{n-1}} \left(\frac{1}{2} T_{1,1,n} - \left(-\frac{1}{2}\right) T_{2,1,n} \right)
\]
and the $(3,2)$ entry is
\[
\frac{1}{2^{n-1}} \left(-\frac{1}{\sqrt{2}} T_{1,2,n} + \frac{1}{\sqrt{2}} T_{2,2,n} \right) .
\]
Using the lemma, we see that every entry of the matrix has norm at most $\frac{1}{\sqrt{2}}\frac{1}{2^{n-1}}$, therefore 
\[
\rho_{\fD_2}(A) = \lim_n \sup_\sigma \left\|\sum_{|w| = n} A^w \otimes \sigma_1(f_{w_1}) \cdots \sigma_n(f_{w_n})\right\|^{1/n} \leq \lim_n \left(\frac{6}{2^{n-1}} \right)^{1/n} = 1/2. 
\]

\end{proof}

\subsection*{Acknowledgements} The authors thank the anonymous referee for several helpful suggestions, in particular for proposing a way to adapt an earlier proof of Theorem \ref{thm:gen_SpecRad} to the infinite-dimensional case. We also thank Michael Hartz and Michael Jury for their feedback and suggestions regarding this point.


\bibliographystyle{plain}
\bibliography{spectral_radius}

\end{document}